\documentclass{amsart}
\usepackage{amsxtra}
\usepackage{mathdots}
\usepackage{mathrsfs}
\usepackage{verbatim}
\usepackage{calc}
\usepackage{graphicx}
\usepackage[breaklinks=true]{hyperref}
\usepackage[shortalphabetic,initials]{amsrefs}
\usepackage{bdsstandard}
\usepackage{pdfsync}

\theoremstyle{plain}
\newtheorem{thm}[equation]{Theorem}
\newtheorem*{thm*}{Theorem}
\newtheorem*{mthm}{Main Theorem}
\newtheorem{prop}[equation]{Proposition}
\newtheorem*{prop*}{Proposition}
\newtheorem{cor}[equation]{Corollary}
\newtheorem*{cor*}{Corollary}
\newtheorem{lem}[equation]{Lemma}
\newtheorem*{lem*}{Lemma}

\newtheorem*{conj*}{Conjecture}

\theoremstyle{definition}

\newtheorem*{defn*}{Definition}

\newtheorem*{eg*}{Example}

\newtheorem*{ex*}{Exercise}
\newtheorem{rk}[equation]{Remark}
\newtheorem*{rk*}{Remark}

\newtheorem*{ntn*}{Notation}

\DeclareMathOperator{\Adm}{Adm}
\DeclareMathOperator{\Conv}{Conv}
\DeclareMathOperator{\Perm}{Perm}

\entrymodifiers={+!!<0pt,\fontdimen22\textfont2>}

\newcommand{\aff}{\ensuremath{\mathrm{a}}\xspace}

\begin{document}

\title[Admissibility and permissibility for minuscule cocharacters]{Admissibility and permissibility for minuscule cocharacters in orthogonal groups}
\author{Brian D. Smithling}
\address{University of Toronto, Department of Mathematics, 40 St.\ George St.,\ Toronto, ON  M5S 2E4, Canada}
\email{bds@math.toronto.edu}

\begin{abstract}
For a given cocharacter $\mu$, $\mu$-admissibility and $\mu$-permissibility are combinatorial notions introduced by Kottwitz and Rapoport that arise in the theory of bad reduction of Shimura varieties.  In this paper we prove that $\mu$-admissibility is equivalent to $\mu$-permissibility in all previously unknown cases of minuscule cocharacters $\mu$ in Iwahori-Weyl groups attached to split orthogonal groups.  This, combined with other cases treated previously by Kottwitz--Rapoport and the author, establishes the equivalence of $\mu$-admissibility and $\mu$-permissibility for all minuscule cocharacters in split classical groups, as conjectured by Rapoport.
\end{abstract}

\maketitle

\section{Introduction}
\numberwithin{equation}{section}

Motivated by considerations arising from the theory of bad reduction of certain Shimura varieties, in \cite{kottrap00} Kottwitz and Rapoport introduced the combinatorial notions of \emph{$\mu$-admissibility} and \emph{$\mu$-permissibility} for elements in any extended affine Weyl group $\wt W$ relative to any cocharacter $\mu$ in $\wt W$.  The \emph{equivalence} of these notions for certain $\mu$ and $\wt W$ has proved to be an important ingredient in showing that the \emph{Rapoport-Zink local model} \cite{rapzink96} (or certain refinements of it) is flat, or at least topologically flat, in many cases; see, for example, work of G\"ortz \cites{goertz01,goertz03,goertz05}, Pappas and Rapoport \cite{paprap05}, and the author \cites{sm09a,sm09c}.

In general, Kottwitz and Rapoport showed that $\mu$-admissibility always implies $\mu$-permissibility, but Haines and Ng\^o \cite{hngo02b} subsequently showed that the converse can fail.  However, again casting an eye towards the theory of Shimura varieties, Rapoport \cite{rap05}*{\s3, p.~283} conjectured
that the converse should always hold at least when $\mu$ is \emph{minuscule} (but see \eqref{rk:sum_minuscules} for further remarks on the conjecture).  The object of this paper is to complete the proof of Rapoport's conjecture for \emph{split classical groups}, or in other words, for $\wt W$ attached to root data involving only types $A$, $B$, $C$, and $D$.  In light of previous work of Kottwitz-Rapoport \cite{kottrap00} and the author \cite{sm09a}, this reduces to proving the conjecture in the previously untreated cases of minuscule cocharacters in split orthogonal groups.

Let
\begin{equation}\label{disp:mu}
   \mu := (1,0,0,\dotsc,0) \in \ZZ^n.
\end{equation}
Then we may regard $\mu$ as a cocharacter of $O_{2n}$ or of $O_{2n+1}$ in the standard way, and $\mu$ is minuscule for either of these groups.  Our result is the following.

\begin{mthm}\label{st:adm=perm}
Let $\wt W$ be the Iwahori-Weyl group (see \s\ref{ss:i-w_gp}) attached to the split group $O_{2n}$ or $O_{2n+1}$.  Then there is an equality of subsets of $\wt W$,
\[
   \Adm(\mu) = \Perm(\mu).
\]
\end{mthm}

We shall give a review of the rest of the terminology in the statement in \s\ref{ss:defs}.
To fix ideas, we shall work throughout with the groups $O_{2n}$ and $O_{2n+1}$ per se, but everything extends readily to the similitude groups $GO_{2n}$ and $GO_{2n+1}$ or other variants.

The pair $(G,\mu)$ for $G = GO_{2n}$ or $GO_{2n+1}$ arises for certain Shimura varieties of Hodge type, and it is expected that the Main Theorem will be relevant to the (as yet completely undeveloped) theory of the corresponding local models.  Although the pair $(G,\mu)$ is not directly related to local models for PEL Shimura varieties, the equality $\Adm(\mu) = \Perm(\mu)$ for $GO_{2n+1}$, where the underlying root system is of type $B_n$, is more-or-less the essential combinatorial ingredient in proving topological flatness of the Pappas--Rapoport spin local models \cite{paprap09} for ramified, quasi-split $GU_{2n}$ in the case of signature $(2n-1,1)$.  We shall give the details of the proof --- and, in fact, the proof of topological flatness for arbitrary signature --- in the forthcoming article \cite{sm10a}.  Somewhat notably, for other signatures (which lead to the consideration of \emph{nonminuscule} coweights for $B_n$) the notions of admissibility and permissibility are typically \emph{not} equivalent, and it is the admissible set that is always the ``good'' set.

The contents of the paper are as follows.  In \s\ref{s:survey} we give a brief survey of the current literature on $\mu$-admissibility and $\mu$-permissibility.  In \s\ref{s:i-w_gps} we review the Iwahori-Weyl groups of the orthogonal groups.  In \s\ref{s:perm} we obtain a formulation of $\mu$-permissibility in the even case in terms of extended alcoves.  We prove the Main Theorem in the even and odd cases in \s\ref{s:thm_even} and \s\ref{s:thm_odd}, respectively.

\subsection*{Acknowledgments}
It is a pleasure to thank Robert Kottwitz for helpful conversation and encouragement and Michael Rapoport for his continuing support and advice.   I also thank the two of them together with Ulrich G\"ortz and Tom Haines for reading and offering comments on a preliminary version of this paper.  I finally thank the referee for carefully reading the paper and offering some helpful corrections and suggestions.

\subsection*{Notation}
We work throughout with respect to a fixed integer $n \geq 2$.  We write $i^* := 2n+1-i$ for any integer $i$.

Given a vector $v$ with entries in \RR, we write $v(i)$ for its $i$th entry, $\Sigma v$ for the sum of its entries, and $v^*$ for the vector with $v^*(i) = v(i^*)$.  Given another vector $w$, we write $v \leq w$ if $v(i) \leq w(i)$ for all $i$.  We write $(a^{(r)}, b^{(s)}, \dotsc)$ for the vector with $a$ repeated $r$ times, followed by $b$ repeated $s$ times, and so on.  We write $\mathbf d$ for the vector $(d,d,\dotsc,d)$.

We write $S_r$ for the symmetric group on $1,\dotsc$, $r$.

\section{Survey of known results}\label{s:survey}
\numberwithin{equation}{subsection}

Before turning to the main body of the article, we give a brief summary of known results in the literature on $\mu$-admissibility and $\mu$-permissibility.  In this section only, we change notation and let $\mu$ denote an arbitrary cocharacter; afterwards we shall resume using $\mu$ for the particular cocharacter \eqref{disp:mu}.  The main references are the papers of Kottwitz and Rapoport \cite{kottrap00}, Haines and Ng\^o \cite{hngo02b}, and the author \cite{sm09a}.

\subsection{Definitions}\label{ss:defs}
Let $(X^*,X_*,R,R^\vee,\Pi)$ be a based root datum.  The \emph{extended affine Weyl group} is the semidirect product $X_* \rtimes W$, where $W$ denotes the Weyl group of the root datum.  The \emph{affine Weyl group $W_\aff$} is the subgroup $Q^\vee \rtimes W$ of $\wt W$, where $Q^\vee := \ZZ R^\vee$ denotes the coroot lattice.  Let $V := X_* \tensor_\ZZ \RR$.  The root datum determines a distinguished alcove $A$ in $V$, namely the alcove contained in the positive Weyl chamber and whose closure contains the origin; and $W_\aff$ is a Coxeter group generated by the reflections across the walls of $A$.  This choice of Coxeter generators determines a Bruhat order on $W_\aff$, which extends to $\wt W$ in the usual way:  $\wt W$ decomposes as a semidirect product $W_\aff \rtimes \Omega$, where $\Omega$ is the subgroup of elements in $\wt W$ that stabilize $A$; and for elements $x\omega$, $x'\omega' \in \wt W$ with $x$, $x' \in W_\aff$ and $\omega$, $\omega' \in \Omega$, we say $x\omega \leq x'\omega'$ if $\omega = \omega'$ and $x \leq x'$ in $W_\aff$.

%
%

Given any cocharacter $\mu \in X_*$, Kottwitz and Rapoport define the $\mu$-admissible and $\mu$-permissible sets in $\wt W$ as follows.  To avoid confusing the action of $W$ on $X_*$ with multiplication in $\wt W$, we write $t_{\mu'}$ when we wish to regard a cocharacter $\mu'$ as an element of $\wt W$.  The \emph{$\mu$-admissible set} is the set
\[
   \Adm(\mu) := \{\, w\in \wt W \mid w \leq t_{\mu'} \text{ for some } \mu' \in W\mu \,\}.
\]
The \emph{$\mu$-permissible set} is the set
\[
   \Perm(\mu) := \{\, w \in \wt W \mid w \equiv t_\mu \bmod W_\aff
                    \text{ and } wa - a \in \Conv(W\mu) \text{ for all } a \in A\,\},
\]
where $\Conv(W\mu)$ is the convex hull in $V$ of the $W$-orbit of $\mu$.  Note that the containment $wa - a \in \Conv(W\mu)$ holds for all $a \in A$ $\iff$ for each facet $F$ of $A$ of minimal dimension, the containment $wa -a \in \Conv(W\mu)$ holds for some $a \in F$.

\subsection{Statements}
The main results on admissibility and permissibility are as follows.

\begin{thm}\label{st:survey}
Let $\wt W$ be the extended affine Weyl group attached to a based root datum.
\begin{enumerate}
\renewcommand{\theenumi}{\roman{enumi}}
\item
   \emph{(Kottwitz-Rapoport \cite{kottrap00}*{11.2})}\, $\Adm(\mu) \subset \Perm(\mu)$ for any cocharacter $\mu$.\ssk
\item\label{it:A}
   \emph{(Haines-Ng\^o \cite{hngo02b}*{3.3}; Kottwitz-Rapoport \cite{kottrap00}*{3.5})}\, If the root datum is of type $A$, then $\Adm(\mu) = \Perm(\mu)$ for any cocharacter $\mu$.\ssk
\item\label{it:GSp}
   \emph{(Haines-Ng\^o \cite{hngo02b}*{10.1}; Kottwitz-Rapoport \cite{kottrap00}*{4.5, 12.4})}\, Let $\wt W$ be the extended affine Weyl group of $GSp_{2n}$.  Then $\Adm(\mu) = \Perm(\mu)$ for any $\mu$ which is a sum of dominant minuscule cocharacters.\ssk
\item\label{it:classical}
   \emph{(Kottwitz-Rapoport \cite{kottrap00}*{3.5, 4.5}; Smithling \cite{sm09a}*{7.6.1}; Main Theorem)}\, Let $\wt W$ be the extended affine Weyl group of $GL_n$, $GSp_{2n}$, $GO^\circ_{2n}$, or $GO^\circ_{2n+1}$.  Then $\Adm(\mu) = \Perm(\mu)$ for any minuscule cocharacter $\mu$.\ssk
\item\label{it:general_classical}
   More generally, suppose the root datum involves only types $A$, $B$, $C$, and $D$.  Then $\Adm(\mu) = \Perm(\mu)$ for any minuscule cocharacter $\mu$.\ssk
\item\label{it:Adm_neq_Perm}
   \emph{(Haines-Ng\^o \cite{hngo02b}*{7.2})}\, Suppose the root datum is irreducible of rank $\geq 4$ and not of type $A$.  Then $\Adm(\mu) \neq \Perm(\mu)$ for every sufficiently regular cocharacter $\mu$.
\end{enumerate}
\end{thm}

\begin{rk}\label{rk:clarification}
In \eqref{it:A} (resp.,\ \eqref{it:GSp}), the result was proved for $GL_n$ and minuscule cocharacters (resp.,\ minuscule cocharacters) by Kottwitz and Rapoport, and in general by Haines and Ng\^o.  In \eqref{it:classical}, the result was proved for $GL_n$ and $GSp_{2n}$ by Kottwitz and Rapoport, and for the orthogonal groups by the author.  In \eqref{it:Adm_neq_Perm}, see \cite{hngo02b}*{proof of 7.2} for the precise meaning of ``sufficiently regular.''
\end{rk}

\begin{proof}[Proof of \eqref{st:survey}]
Only \eqref{it:classical} and \eqref{it:general_classical} require comment beyond the cited references.  In \eqref{it:classical}, we just need the following trivial remark on orthogonal groups:  with respect to the standard choices of simple roots, the dominant minuscule cocharacters for $GO_{2n}^\circ$ consist of the cocharacters $\mu + \mathbf d$ for $d$ any integer and $\mu$ one of the cocharacters
\[
   (1,0^{(2n-2)},-1),\quad (1^{(n)},0^{(n)}),\quad(1^{(n-1)},0,1,0^{(n-1)})
\]
(the first of these corresponds to the cocharacter denoted $\mu$ in the Introduction, and the other two are treated in \cite{sm09a}); the dominant minuscule cocharacters for $GO_{2n+1}^\circ$ consist of the cocharacters
\[
   (1,0^{(2n-1)},-1) + \mathbf d
\]
for $d$ any integer; and for either group, for any cocharacter $\mu$, the equality $\Adm(\mu) = \Perm(\mu)$ holds $\iff$ the equality $\Adm( \mu + \mathbf 1 ) = \Perm( \mu + \mathbf 1 )$ holds.

Part \eqref{it:general_classical} reduces to \eqref{it:classical} in a straightforward way.  Let $\R = (X^*,X_*,R,R^\vee,\Pi)$ be a based root datum, and consider the root datum
\[
   \R' := (Q,P^\vee,R,R^\vee,\Pi),
\]
where $Q \subset X^*$ is the root lattice and $P^\vee$ is the coweight lattice in the vector space $V := \Hom_\RR(Q\tensor_\ZZ\RR, \RR)$.  The cocharacters $X_*$ map naturally to $V$ with kernel $Z := \{x\in X_* \mid \langle x,\alpha\rangle = 0 \text{ for all } \alpha \in R\}$ and image contained in $P^\vee$.  The root data \R and $\R'$ have canonically isomorphic Weyl groups and affine Weyl groups, and we get a natural homomorphism of extended affine Weyl groups $\pi \colon \wt W_\R \to \wt W_{\R'}$ with kernel $Z$; here and below we use subscripts \R and $\R'$ to denote objects attached to $\R$ and $\R'$, respectively.  Given a cocharacter $\mu \in X_*$, let $\ol\mu$ denote its image in $P^\vee$.  Then $\pi$ carries $\Adm_\R(\mu) \isoarrow \Adm_{\R'}(\ol\mu)$ and $\Perm_\R(\mu) \isoarrow \Perm_{\R'}(\ol\mu)$; and moreover we have $\Adm_\R(\mu) = \Perm_\R(\mu)$ if and only if $\Adm_{\R'}(\ol\mu) = \Perm_{\R'}(\ol\mu)$.  Now, on the one hand, for each of the root data \R in the statement of \eqref{it:classical}, $X_*$ maps \emph{onto} $P^\vee$, and we therefore find an equality between admissible and permissible sets for all minuscule coweights in the reduced, irreducible root systems of types $A$, $B$, $C$, and $D$.  On the other hand, taking \R to be the given root datum in \eqref{it:general_classical}, we get $\wt W_{\R'} \ciso \prod_i P^\vee_i \rtimes W_i$, where the product runs through the irreducible components $R_i$ of $R$, each with coweight lattice $P^\vee_i$ and Weyl group $W_i$.  It follows from what we have just said that admissibility and permissibility are equivalent for all minuscule coweights in each factor, and this immediately implies the same in $\wt W_{\R'}$.
\end{proof}

\begin{rk}
More generally, one may consider $\mu$-admissibility and $\mu$-permissi\-bil\-ity in the \emph{Iwahori-Weyl group} attached to any connected reductive group over a discretely valued field; see \citelist{\cite{rap05}*{\s\s2--3}\cite{paprap09}*{\s\s2.1--2.2}}.  In this way, extended affine Weyl groups attached to root data correspond to \emph{split} groups.  As noted by Rapoport, the sets $\Adm(\mu)$ and $\Perm(\mu)$ can fail to be equal, even for minuscule $\mu$, for nonsplit groups.
\end{rk}

\begin{rk}
As noted in the Introduction, and in addition to the examples of Haines and Ng\^o in \eqref{it:Adm_neq_Perm}, in \cite{sm10a} we shall give some explicit examples of cocharacters $\mu$ in rank $n$ root data of type $B_n$, $n \geq 3$, for which $\Adm(\mu) \neq \Perm(\mu)$.
\end{rk}

\begin{rk}\label{rk:sum_minuscules}
In \cite{rap05}*{\s3, p.~283}, Rapoport actually formulates a more optimistic conjecture than the one cited in the Introduction, namely, that the equality $\Adm(\mu) = \Perm(\mu)$ should hold in any extended affine Weyl group whenever $\mu$ is a sum of dominant minuscule cocharacters.  We shall show by explicit example in \cite{sm10a} that this more optimistic conjecture is false.
\end{rk}

\begin{rk}
There is another variant of permissibility, called \emph{strong permissibility}, due to Haines and Ng\^o.  They prove that strong $\mu$-permissibility always implies $\mu$-admissibility, and that these notions, along with the notion of $\mu$-permissibility, are all equivalent for root data of type $A$.  It would be interesting to clarify the situation for other root types.
\end{rk}

\section{Iwahori-Weyl groups of orthogonal groups}\label{s:i-w_gps}

Rather than working directly with the extended affine Weyl groups of the connected groups $SO_{2n}$ and $SO_{2n+1}$ (say over a field of characteristic not $2$), we shall allow ourselves the harmless added bit of generality of working with the Iwahori-Weyl groups of the full orthogonal groups.  We review the relevant material in this section.

\subsection{Iwahori-Weyl group}\label{ss:i-w_gp}
For any split (but not necessarily connected) reductive group $G$ with split maximal torus $T$, the \emph{Iwahori-Weyl group} is the semidirect product $\wt W_G := X_*(T) \rtimes W_G$, where $W_G$ is the Weyl group.  As in the previous section, we write $t_\nu$ for the cocharacter $\nu$ regarded as an element of $\wt W_G$, and we refer to such elements as \emph{translation elements} in $\wt W_G$.  For $w = t_\nu\sigma \in \wt W_G$ with $\sigma \in W_G$, we refer to $t_\nu$ as the \emph{translation part} of $w$ and $\sigma$ as the \emph{linear part} of $w$.

We shall identify the Iwahori-Weyl groups of the orthogonal groups $O_{2n}$ and $O_{2n+1}$ to the common group $\ZZ^n \rtimes S^*_{2n}$,
where $S^*_{2n} \subset S_{2n}$ is the subgroup of permutations $\sigma$ satisfying $\sigma(i^*) = \sigma(i)^*$ for all $i \in\{1,\dotsc,2n\}$.  The group $S_{2n}^*$ decomposes as a semidirect product $\{\pm 1\}^n \rtimes
S_n$, and here $S_n$ and $\{\pm 1\}^n$ act respectively on $\ZZ^n$ by permuting entries and by factorwise multiplication.  We shall adopt the separate notations $\wt W_{2n}$ and $\wt W_{2n+1}$ for the Iwahori-Weyl group in the even and odd cases, respectively; although these are identified as groups, they will carry different Bruhat orders.

In the even case, inside $\wt W_{2n}$ is the Iwahori-Weyl group $\wt W_{2n}^\circ$ of $SO_{2n}$,
\[
   \wt W_{2n}^\circ := \ZZ^n \rtimes S^\circ_{2n},
\]
where $S^\circ_{2n} \subset S^*_{2n}$ is the subgroup of elements which are even in $S_{2n}$.  The \emph{affine Weyl group $W_{\aff,2n}$} is the subgroup
\[
   W_{\aff,2n} := Q_{2n}^\vee \rtimes S_{2n}^\circ,
\]
where
\begin{equation}\label{disp:Q_2n^vee}\textstyle
   Q_{2n}^\vee := \bigl\{\, (x_1,\dotsc,x_n) \in \ZZ^n \bigm| \sum_{i=1}^n x_i \text{ is even} \,\bigr\}
\end{equation}
is the coroot lattice.  The group $\wt W_{2n}^\circ$ is the extended affine Weyl group attached to the root datum of $SO_{2n}$ in the sense of \s\ref{ss:defs}.

The odd case is a little simpler:  $O_{2n+1}$ and $SO_{2n+1}$ have common Iwahori-Weyl group $\wt W_{2n+1}$, and this is the extended affine Weyl group attached to the root datum of $SO_{2n+1}$ in the sense of \s\ref{ss:defs}.  The \emph{affine Weyl group $W_{\aff,2n+1}$} is the subgroup
\[
   W_{\aff,2n+1} := Q_{2n+1}^\vee \rtimes S_{2n}^*,
\]
where $Q_{2n+1}^\vee$ is defined just as in the even case \eqref{disp:Q_2n^vee}.

\subsection{Roots}
The \emph{roots} of $SO_{2n}$ consist of the maps
\[
   (x_1,\dotsc,x_n) \mapsto \pm x_i \pm x_j
   \quad\text{for}\quad
   1 \leq i < j \leq n.
\]
We take the $n$ roots
\[
   (x_1,\dotsc,x_n) \mapsto
   \begin{cases}
      x_i - x_{i+1}, & 1 \leq i \leq n-1;\\
      x_{n-1} + x_n
   \end{cases}
\]
as simple roots.

The \emph{roots} of $SO_{2n+1}$ consist of the maps
\[
   (x_1,\dotsc,x_n) \mapsto 
   \begin{cases}
      \pm x_i,  &  1 \leq i \leq n;\\
      \pm x_i \pm x_j,  &  1 \leq i < j \leq n.
   \end{cases}
\]
We take the $n$ roots
\[
   (x_1,\dotsc,x_n) \mapsto
   \begin{cases}
      x_i - x_{i+1}, & 1 \leq i \leq n-1;\\
      x_n
   \end{cases}
\]
as simple roots.

\subsection{Base alcoves}
As in \s\ref{ss:defs}, our choices of simple roots determine base alcoves.  For $SO_{2n}$, the base alcove $A$ is the alcove in $\RR^n$ with $n+1$ vertices
\begin{align*}
   a_0 &:= (0,\dotsc,0),\\
   a_{0'} &:= (1,0^{(n-1)}),\\
   a_i &:= \bigl((\tfrac 1 2)^{(i)},0^{(n-i)}\bigr),\quad 2 \leq i \leq n-2,\\
   a_n &:= \bigl(\tfrac 1 2,\dotsc,\tfrac 1 2\bigr),\\
   a_{n'} &:= \bigl((\tfrac 1 2)^{(n-1)},-\tfrac 1 2\bigr).
\end{align*}
For $SO_{2n+1}$, the base alcove is the alcove in $\RR^n$ with $n+1$ vertices
\[
   (0,\dotsc,0),\quad (1,0^{(n-1)}),\quad \text{and}\quad \bigl((\tfrac 1 2)^{(i)},0^{(n-i)}\bigr)\quad \text{for}\quad  2 \leq i \leq n.
\]

\subsection{Bruhat order}
The group $\wt W_{2n+1}$ is endowed with a Bruhat order as described for any extended affine Weyl group in \s\ref{ss:defs}.  In the even case, although $\wt W_{2n}$ is not the extended affine Weyl group attached to a root datum, we still get a Bruhat order on it in the way described in \s\ref{ss:defs}:  $\wt W_{2n}$ decomposes as a semidirect product $W_{\aff,2n} \rtimes \Omega$, where $\Omega$ denotes the stabilizer of $A$ in $\wt W_{2n}$, and everything carries over as usual.

\section{Permissibility for even orthogonal groups}\label{s:perm}
In this section we translate the notion of $\mu$-permissibility in $\wt W_{2n}$ into a combinatorial notion expressible in terms of extended alcoves.

\subsection{Convex hull}
Recall the cocharacter $\mu$ from \eqref{disp:mu}.  The relevant convex hull that comes up in the definition of $\mu$-permissibility is $\Conv(S_{2n}^\circ \mu)$, which one readily verifies to equal
\begin{equation}\label{disp:conv_hull}\textstyle
   \bigl\{\, (x_1,\dotsc,x_n) \in \RR^n \bigm|  \sum_{i=1}^n |x_i| \leq 1 \,\bigr\}.
\end{equation}

\subsection{Extended alcoves}\label{ss:extd_alcs}
We next recall the notion of extended alcove of Kottwitz and Rapoport  \cite{kottrap00}; see also \cite{sm09a}.  We'll choose our base alcove so that we follow the conventions in \cite{kottrap00}, which are ``opposite'' to those in \cite{sm09a}.  

To define extended alcoves, we'll make use of the Iwahori-Weyl group $\wt W_{GL_{2n}}$ of $GL_{2n}$,
\[
   \wt W_{GL_{2n}} := \ZZ^{2n} \rtimes S_{2n},
\]
where $S_{2n}$ acts on $\ZZ^{2n}$ by permuting entries.  There is a natural embedding $\wt W_{2n} \inj \wt W_{GL_{2n}}$ defined on translation elements by
\begin{equation}\label{disp:cochar_emb}
   (x_1,\dotsc,x_n) \mapsto (x_1,\dotsc,x_n,-x_n,\dotsc,-x_1)
\end{equation}
and on finite Weyl group elements by the inclusion $S_{2n}^* \subset S_{2n}$.  For the rest of this section and \s\ref{s:thm_even}, we shall often tacitly identify elements and subsets of $\RR^n$ with their image in $\RR^{2n}$ via the map \eqref{disp:cochar_emb}.

An \emph{extended alcove} for $\wt W_{2n}$ is a sequence $v_0,\dotsc$, $v_{2n-1}$ of elements in $\ZZ^{2n}$ such that, putting $v_{2n} := v_0 + \mathbf 1$,
\begin{enumerate}
\renewcommand{\theenumi}{A\arabic{enumi}}
\item
   $v_0 \leq v_1 \leq \dotsb \leq v_{2n}$;
\item
   $\Sigma v_k = \Sigma v_{k-1} + 1$ for all $1 \leq k \leq 2n$; and
\item\label{it:dual_cond}
   $v_k + v_{2n-k}^* = \mathbf 1$ for all $1 \leq k \leq 2n$.
\end{enumerate}
Note that condition \eqref{it:dual_cond} requires $v_0$ to lie in the image of \eqref{disp:cochar_emb}.  The sequence of elements $\omega_k := (1^{(k)},0^{(2n-k)})$ is an extended alcove, which we call the \emph{standard extended alcove}.  The group $\wt W_{2n}$ acts simply transitively on extended alcoves via its embedding into $\wt W_{GL_{2n}}$, and we identify $\wt W_{2n}$ with the set of extended alcoves by taking the standard extended alcove as base point.

\subsection{The vector \texorpdfstring{$\mu_k$}{mu-k}}
Given an extended alcove $v_0,\dotsc$, $v_{2n-1}$ for $\wt W_{2n}$, let
\begin{equation}\label{disp:mu_k}
   \mu_k := v_k - \omega_k, \quad 0 \leq k \leq 2n.
\end{equation}
Then $\Sigma \mu_k = 0$, we have the periodicity relation $\mu_0 = \mu_{2n}$, and condition \eqref{it:dual_cond} is equivalent to
\begin{equation}\label{disp:mu_dlty_cond}
   \mu_k + \mu_{2n-k}^* = \mathbf 0 \quad \text{for all}\quad 0 \leq k \leq 2n.
\end{equation}

Note that if the extended alcove corresponds to the element $w \in \wt W_{2n}$ with linear part $\sigma \in S_{2n}^*$, then
\begin{equation}\label{disp:mu_recursion}
   \mu_{k} = \mu_{k-1} + e_{\sigma(k)} - e_k \quad\text{for all}\quad 1 \leq k \leq 2n,
\end{equation}
where $e_1,\dotsc$, $e_{2n}$ denotes the standard basis of $\ZZ^{2n}$.

\subsection{Basic inequalities}
A priori, the $\mu_k$'s of the previous subsection are not directly suited to questions of permissibility in $\wt W_{2n}$, since $\omega_k$ is not a vertex of $A$ for $k \neq 0$ (and indeed is not even in the image of \eqref{disp:cochar_emb} for such $k$).  We'll use the following basic inequalities to help forge a better link.  For $0 \leq k \leq n$, let
\[
   A_k := \{1,2,\dotsc,k,k^*,k^*+1,\dotsc,2n\}
   \quad\text{and}\quad
   B_k := \{k+1,k+2,\dotsc,2n-k\}.
\]

\begin{lem}\label{st:basic_ineqs}
For $0 \leq k \leq n$, we have
\[
   -1 \leq \mu_k(i) + \mu_k(i^*) \leq 0 \quad\text{for}\quad i \in A_k
\]
and
\[
   0 \leq \mu_k(i) + \mu_k(i^*) \leq 1 \quad\text{for}\quad i \in B_k.
\]
\end{lem}

\begin{proof}
Let $\sigma \in S_{2n}^*$ denote the linear part of $w$.  Since $\sigma$ is a permutation, it follows from \eqref{disp:mu_recursion} that
\[
   \mu_k(i) \leq \mu_{2n-k}(i) \leq \mu_k(i) + 1 \quad\text{if}\quad i \in A_k
\]
and
\[
   \mu_k(i) \geq \mu_{2n-k}(i) \geq \mu_k(i) - 1 \quad\text{if}\quad i \in B_k.
\]
Now use $\mu_{2n-k}(i) = - \mu_k(i^*)$ \eqref{disp:mu_dlty_cond} in each case to obtain the asserted inequalities.
\end{proof}

\subsection{The vector \texorpdfstring{$\nu_k$}{nu-k}}
For $0 \leq k \leq 2n$, let
\begin{equation}\label{disp:nu_k}
   \nu_k := \frac{\mu_k + \mu_{2n-k}}2.
\end{equation}
Then $\nu_k \in \frac 1 2 \ZZ^{2n}$, $\nu_k = \nu_{2n-k}$, and
\[
   \nu_0 = \mu_0 \quad\text{and}\quad \nu_n = \mu_n.
\]
Moreover, \eqref{disp:mu_dlty_cond} allows us to also write
\begin{equation}\label{disp:nu_k_mu_k}
   \nu_k = \frac{\mu_k - \mu_k^*}2.
\end{equation}
Note that, since $\mu_k$ has integer entries, \eqref{disp:nu_k_mu_k} and the basic inequalities \eqref{st:basic_ineqs} allow us to recover $\mu_k$ uniquely from $\nu_k$.

The $\nu_k$'s are well-suited to questions of permissibility.  Indeed, let us say that that $\nu_k$ arises from the element $w \in \wt W_{2n}$, and let
\begin{equation}\label{disp:a_k'}
   a_k' := \frac{\omega_k + \omega_{2n-k} - \mathbf 1}2 = \bigl((\tfrac 1 2)^{(k)}, 0^{(2n-2k)},(-\tfrac 1 2)^{(k)}\bigr), \quad 0 \leq k \leq n.
\end{equation}
Then $a_k'$ equals the vertex $a_k$ for $k \neq 1$, $n-1$; $a_k'$ is contained in the closure of $A$ for every $k$; and
\begin{equation}\label{disp:wa-a_nu_k}
    \nu_k = wa_k' - a_k'.
\end{equation}

We also note that the recursion relation \eqref{disp:mu_recursion} for the $\mu_k$'s and \eqref{disp:nu_k_mu_k} yield the recursion relation
\begin{equation}\label{disp:nu_recursion}
   \nu_k = \nu_{k-1} + \frac{e_{\sigma(k)}-e_k+e_{k^*}-e_{\sigma(k)^*}}2,
\end{equation}
with the notation as in \eqref{disp:mu_recursion}.

\subsection{\texorpdfstring{$\mu$}{mu}-permissibility}
We now translate the condition of $\mu$-permissibility into a combinatorial condition on extended alcoves.  Given $w \in \wt W_{2n}$, we write $\sigma \in S_{2n}^*$ for its linear part; $v_0,\dotsc$, $v_{2n-1}$ for its extended alcove; and $\mu_k$ and $\nu_k$ for the vectors \eqref{disp:mu_k} and \eqref{disp:nu_k}, respectively.  We begin with a lemma.

\begin{lem}\label{st:sigma_even}
$\sigma$ is even $\iff$ $\mu_0 \equiv \mu_n \bmod Q^\vee_{2n}$.
\end{lem}

\begin{proof}
Combine the facts that $\sigma$ is even or odd according as
\[
   \#\bigl\{\, i \in \{1,\dotsc,n\} \mid \sigma(i) > n \,\bigr\}
\]
is even or odd, and that, by \eqref{disp:mu_recursion},
\[
   \mu_n = \mu_0 + \sum_{i=1}^n e_{\sigma(i)} - \sum_{i=1}^n e_i.\qedhere
\]
\end{proof}


We are now ready for the main result of \s\ref{s:perm}.  Via \eqref{disp:cochar_emb}, we identify $\mu$ with
\[
   (1,0^{(2n-2)},-1) \in \ZZ^{2n}.
\]
Our statement involves the Weyl orbits $S_{2n}^\circ \mu = S_{2n}^*\mu$ and $S_{2n}\mu$ in $\ZZ^{2n}$. Note that the first of these is contained in the image of the map \eqref{disp:cochar_emb}, whereas the second is not.

\begin{prop}\label{st:comb_perm}
The following are equivalent.
\begin{enumerate}
\item\label{it:w_mu-perm}
   $w$ is $\mu$-permissible.
\item\label{it:nu_k_perm_cond}
   $\nu_k \in \Conv(S_{2n}^\circ\mu)$ for all $0 \leq k \leq n$, and furthermore $\nu_0$, $\nu_n \in S_{2n}^\circ \mu$.
\item\label{it:mu_k_perm_cond}
   $\mu_k \in S_{2n}\mu \cup\{\mathbf 0\}$ for all $0 \leq k \leq 2n$, and furthermore $\mu_0$, $\mu_n \in S_{2n}^\circ \mu$.
\end{enumerate}
\end{prop}

\begin{proof}
The equivalence \eqref{it:nu_k_perm_cond} $\Longleftrightarrow$ \eqref{it:mu_k_perm_cond} follows from the explicit expression \eqref{disp:conv_hull} for $\Conv(S_{2n}^\circ\mu)$, the basic inequalities \eqref{st:basic_ineqs}, and \eqref{disp:nu_k_mu_k}.

We now turn to \eqref{it:w_mu-perm} $\Longleftrightarrow$ \eqref{it:nu_k_perm_cond}.  First note that we have equivalences
\begin{equation}\label{disp:equivs}
\begin{aligned}
   w \equiv t_\mu \bmod W_{\aff,2n}
      & \iff \text{$\sigma$ is even and $\nu_0 \equiv \mu \bmod Q^\vee_{2n}$}\\
      & \iff \nu_0 \equiv \nu_n \equiv \mu \bmod Q^\vee_{2n},
\end{aligned}
\end{equation}
where the second equivalence uses \eqref{st:sigma_even}.  This, combined with \eqref{disp:wa-a_nu_k}, renders the implication \eqref{it:w_mu-perm} $\Longrightarrow$ \eqref{it:nu_k_perm_cond} transparent.

Conversely, suppose \eqref{it:nu_k_perm_cond}.  Then $w \equiv t_\mu \bmod W_{\aff,2n}$ by \eqref{disp:equivs}, and for each vertex $a_k$ of $A$ with $k = 0$, $2$, $3,\dotsc$, $n-2$, $n$,
\[
   wa_k - a_k = wa_k' - a_k' = \nu_k \in \Conv(S^\circ_{2n}\mu).
\]
So it remains to show $wa_k - a_k \in \Conv(S^\circ_{2n}\mu)$ for $k = 0'$, $n'$.

For $k = 0'$, it follows from the definitions that
\begin{align*}
   wa_{0'} - a_{0'} &= \nu_0 + e_{\sigma(1)} - e_1 + e_{2n} - e_{\sigma(2n)},\\
   \nu_1 &= \nu_0 + \frac{e_{\sigma(1)} - e_1 + e_{2n} - e_{\sigma(2n)}}2.
\end{align*}
By assumption,  $\nu_0 \in S_{2n}^\circ \mu$ is of the form $e_j - e_{j^*}$ for some $j \in \{1,\dotsc,2n\}$.  The containment $\nu_1 \in \Conv(S^\circ_{2n}\mu)$ and the explicit expression for $\Conv(S^\circ_{2n}\mu)$ then force $\sigma(1) = 1$, $j = 1$, or $j = \sigma(2n)$.  In each case we visibly have $wa_{0'} - a_{0'} \in \Conv(S^\circ_{2n}\mu)$.

The argument for $k = n'$ goes similarly, with $\nu_n$ in the role of $\nu_0$ and $\nu_{n-1}$ in the role of $\nu_1$.
\end{proof}

\begin{rk}\label{rk:perm_nu}
It follows from \eqref{disp:nu_k_mu_k} and condition \eqref{it:mu_k_perm_cond} in the proposition that when $w$ is $\mu$-permissible, $\nu_k$ is of the form $\frac 1 2 (e_i - e_{j} + e_{j^*} - e_{i^*})$ for some $i$, $j \in \{1,\dotsc,2n\}$.
\end{rk}

\subsection{Remarks on the odd case}
Essentially everything in this section carries over to the odd Iwahori-Weyl group $\wt W_{2n+1}$ as well.  In particular, the notions of extended alcove, $\mu_k$, and $\nu_k$ are defined in exactly the same way (after all, $\wt W_{2n+1}$ and $\wt W_{2n}$ have the same underlying group), and one obtains the following analogue of \eqref{st:comb_perm}.

\begin{prop}
Let $w \in \wt W _{2n+1}$.  The following are equivalent.
\begin{enumerate}
\item
   $w$ is $\mu$-permissible in $\wt W_{2n+1}$.
\item
   $\nu_k \in \Conv(S_{2n}^*\mu)$ for all $0 \leq k \leq n$, and futhermore $\nu_0 \in S_{2n}^* \mu$.
\item
   $\mu_k \in S_{2n}\mu \cup\{\mathbf 0\}$ for all $0 \leq k \leq 2n$, and furthermore $\mu_0 \in S_{2n}^* \mu$.\qed
\end{enumerate}
\end{prop}

Since we won't actually need this proposition later on, we leave the details to the reader.

\section{Main theorem for even orthogonal groups}\label{s:thm_even}
In this section we prove the Main Theorem from the Introduction for even orthogonal groups.  The broad strategy of our proof is the same as that of Kottwitz and Rapoport \cite{kottrap00}; see also \cite{sm09a}.  Namely, since $\mu$-admissibility implies $\mu$-permissibility for any cocharacter $\mu$ in any extended affine Weyl group, we must show that the converse holds for our particular cocharacter $\mu$, which we continue to identify with $(1,0^{(2n-2)},-1)$ via \eqref{disp:cochar_emb}.  For this, let $w \in \wt W_{2n}$ be $\mu$-permissible.  If $w$ is a translation element, then necessarily $w$ is an $S_{2n}^\circ$-conjugate of $t_\mu$ and there is nothing further to do.  So we might as well assume $w$ is not a translation element.  In this case, our task is to find an affine root $\wt\alpha$ such that, for the associated reflection $s_{\wt\alpha} \in W_{\aff,2n}$, $s_{\wt\alpha} w$ is again $\mu$-permissible and $w < s_{\wt\alpha} w$ in the Bruhat order.  Repeating the argument as needed, we then obtain a chain $w < s_{\wt\alpha} w < \dotsb$ of $\mu$-permissible elements which must terminate in a translation element since $\Perm(\mu)$ is manifestly a finite set.

We shall spend the subsections \ref{ss:refl}--\ref{ss:refl_bo} laying the groundwork to implement our strategy, with the proof proper coming in \s\ref{ss:pf_even}.

\subsection{Reflections}\label{ss:refl}
Let us write $X_*$ for the image of $\ZZ^n$ in $\ZZ^{2n}$ under the map \eqref{disp:cochar_emb}, and consider the affine linear function
\[
   \wt\alpha_{i,j;d}\colon
   \xymatrix@R=0ex{
      X_* \ar[r] & \ZZ\vphantom{X_*}\\
      (x_1,\dotsc,x_{2n})\vphantom{x_i - x_j - d} \ar@{|->}[r] & x_i - x_j - d \vphantom{(x_1,\dotsc,x_{2n})}
   }
\]
for $i$, $j \in \{1,\dotsc,2n\}$ with $i < j$ and $d\in\ZZ$.  Then $\wt\alpha := \wt\alpha_{i,j;d}$ is an affine root of $SO_{2n}$ precisely when $j \neq i^*$, and up to sign, all affine roots are obtained in this way.  Plainly $\wt\alpha_{i,j;d} = \wt\alpha_{j^*,i^*;d}$.  Attached to $\wt\alpha$ is its linear part $\alpha:= \wt\alpha_{i,j;0}$; the coroot $\alpha^\vee = e_i - e_j + e_{j^*} - e_{i^*}$, where $e_1,\dotsc,$ $e_{2n}$ denotes the standard basis in $\ZZ^{2n}$; and the reflection $s_{\wt\alpha} = s_{i,j;d} \in W_{\aff,2n}$ which acts on $X_* \tensor_\ZZ \RR$ by the rule $s_{\wt\alpha} x = x - \langle \wt\alpha, x \rangle \alpha^\vee$.  Visually, in the case $i < j < j^* < i^*$,
\begin{multline*}
   (\dotsc,x_i,\dotsc,x_j,\dotsc,x_{j^*},\dotsc,x_{i^*}\dotsc)
   \overset{s_\alpha}{\mapsto}\\
   (\dotsc,x_j+d,\dotsc,x_i-d,\dotsc,x_{i^*} + d,\dotsc,x_{j^*}-d,\dotsc).
\end{multline*}

If $w\in\wt W_{2n}$ has extended alcove $v_0,\dotsc,v_{2n-1}$, then $s_{\wt\alpha} w$ has extended alcove $s_{\wt\alpha} v_0,\dotsc,s_{\wt\alpha} v_{2n-1}$.

\subsection{Reflections and permissibility}\label{ss:refl_perm}
Given an affine root $\wt\alpha$ and a $\mu$-permissible $w \in \wt W_{2n}$, our aim in this subsection is to investigate when $s_{\wt\alpha} w$ is again $\mu$-permissible.  This subsection may be regarded as analogous to \cite{kottrap00}*{\s5.2} or to \cite{sm09a}*{\s8.4}, but our inspiration is taken from \cite{kottrap00}*{\s11} instead.  The statements and arguments in this and the next subsection are general and hold in an obvious way for any extended affine Weyl group and any cocharacter, but rather than introduce new notation, we shall express them in terms of our particular group $\wt W_{2n}$ and particular cocharacter $\mu$.

Let $\alpha$ denote the linear part of $\wt\alpha$.  Since $w$ is $\mu$-permissible, we see from \eqref{st:comb_perm} that $s_{\wt\alpha}w$ is $\mu$-permissible $\iff$ for all $0 \leq k \leq n$,
\[
   s_{\wt\alpha}w a_k' - a_k' \in \Conv(S_{2n}^\circ\mu),
\]
with $a_k'$ defined in \eqref{disp:a_k'}.  Applying $s_\alpha$ to $s_{\wt\alpha}w a_k' - a_k'$ and using that $\Conv(S_{2n}^\circ\mu)$ is stable under the action of the Weyl group, we conclude the following.

\begin{lem}\label{st:perm_cond}
If $w$ is $\mu$-permissible, then $s_{\wt\alpha} w$ is $\mu$-permissible $\iff$
\begin{flalign*}
   \phantom{\qed}& & wa_k' - a_k' + \langle \wt\alpha, a_k' \rangle \alpha^\vee \in \Conv( S_{2n}^\circ \mu)\quad \text{for all}\quad 0 \leq k \leq n. & & \qed
\end{flalign*}
\end{lem}

Note that if $\wt\alpha = \wt\alpha_{i,j;d}$ with $i < j \neq i^*$, then the coefficient
\[
   \langle \wt\alpha, a_k'\rangle = \langle \alpha, a_k'\rangle -d \in \bigl\{0,\tfrac 1 2, 1\bigr\} - d.
\]

\subsection{Reflections and the Bruhat order}\label{ss:refl_bo}
We continue with $w \in \wt W_{2n}$ and affine root $\wt\alpha$ with linear part $\alpha$.  The elements $w$ and $s_{\wt\alpha} w$ are related in the Bruhat order, and in this subsection we wish to investigate which way the inequality goes in terms of $\wt\alpha$ and the extended alcove attached to $w$.  We again root our discussion in \cite{kottrap00}*{\s11}.

To begin, we recall that for any $v \in X_* \tensor \RR$, the vectors
\begin{equation}\label{disp:wv-v_variants}
\begin{aligned}
   wv-v&,\\
   s_{\wt\alpha}wv-v &= wv -v - \langle\wt\alpha, wv \rangle \alpha^\vee,\\
   s_\alpha(wv-v) &= wv-v - \langle\alpha, wv -v \rangle \alpha^\vee
\end{aligned}
\end{equation}
all lie along the same line, and that  \cite{kottrap00}*{11.4} 
$w < s_{\wt\alpha}w$ $\iff$ for some $v \in \ol A$, $s_{\wt\alpha}wv-v$ lies outside the segment of this line between $wv-v$ and $s_\alpha(wv-v)$.  Since the line has direction vector $\alpha^\vee$ and the midpoint of the segment lies in the hyperplane where $\alpha$ vanishes, this holds $\iff$
\[
   |\langle\alpha, wv -v\rangle| 
      < |\langle\alpha, s_{\wt\alpha}wv -v\rangle|
      \quad\text{for some} \quad v \in \ol A;
\]
or, replacing $s_{\wt\alpha}wv -v$ by the reflected vector $s_{\alpha}( s_{\wt\alpha} wv -v)$, $\iff$
\[
   |\langle\alpha, wv -v\rangle| 
      < \bigl|\bigl\langle\alpha, wv -v + \langle \wt \alpha, v \rangle \alpha^\vee \bigr\rangle\bigr| 
      \quad\text{for some}\quad v \in \ol A.
\]
To avoid having to consider all $v \in \ol A$, we now reformulate our criterion in a way that will allow us to consider only the points $a_k'$.

\begin{lem}\label{st:bo}
Let $S$ be any subset of $\ol A$, and suppose that
\begin{enumerate}
\item\label{it:bo1}
   $|\langle\alpha, wv -v\rangle| < \bigl|\bigl\langle\alpha, wv -v + \langle \wt \alpha, v \rangle \alpha^\vee \bigr\rangle\bigr|$ for some $v \in S$; or
\item\label{it:bo2}
   there exists $v \in S$ such that $\langle \wt \alpha, v \rangle = 0$ and $\langle \alpha, wv-v \rangle$ is positive or negative according as $\wt \alpha$ is positive or negative on $A$.
\end{enumerate}
Then $w < s_{\wt\alpha}w$.  The converse holds if $S$ is not contained in the closure of a single proper subfacet of $A$.
\end{lem}

Note that conditions \eqref{it:bo1} and \eqref{it:bo2} are not mutually exclusive.

\begin{proof}[Proof of \eqref{st:bo}]
We have just seen that \eqref{it:bo1} $\Longrightarrow$ $w < s_{\wt\alpha} w$.  To see that \eqref{it:bo2} $\Longrightarrow$ $w < s_{\wt\alpha} w$, note that for any $v$,
\begin{equation}\label{disp:alpha_wtalpha}
   \langle \alpha, wv-v \rangle 
      = \langle \wt\alpha, wv \rangle - \langle \wt\alpha, v \rangle.
\end{equation}
The conclusion now follows in case $v \in \ol A$ and $\langle \wt\alpha, v \rangle = 0$, since we have $w < s_{\wt\alpha}w$ exactly when $A$ and $wA$ lie on the same side of the hyperplane where $\wt \alpha$ vanishes.

To see that the converse holds when $S$ is not contained in the closure of a single proper subfacet, we may assume by the above discussion that $|\langle\alpha, wv -v\rangle| = \bigl|\bigl\langle\alpha, wv -v + \langle \wt \alpha, v \rangle \alpha^\vee \bigr\rangle\bigr|$ for all $v \in S$.  Hence for all $v \in S$,
\[
   wv -v + \langle \wt \alpha, v \rangle \alpha^\vee = wv - v
\]
or
\[
   wv -v + \langle \wt \alpha, v \rangle \alpha^\vee 
      = s_{\alpha}(wv-v)
      = wv - v -[\langle \wt \alpha, wv \rangle - \langle \wt\alpha, v \rangle]\alpha^\vee.
\]
The first of these is equivalent to $\langle \wt \alpha, v \rangle = 0$, and the second to $\langle \wt \alpha, wv \rangle = 0$.  Now, our assumption on $S$ implies that there exists $v \in S$ for which $\langle \wt \alpha, wv \rangle \neq 0$.  Hence $\langle \wt \alpha, v \rangle = 0$, and, again using \eqref{disp:alpha_wtalpha}, we deduce \eqref{it:bo2}.
\end{proof}

\subsection{Proof of the theorem}\label{ss:pf_even}
As explained at the beginning of the section, the proof of the Main Theorem for even orthogonal groups follows from the following result.

\begin{prop}\label{st:strong_thm_even}
If $w \in \wt W_{2n}$ is $\mu$-permissible and not a translation element, then there exists an affine reflection $s_{\wt\alpha}$ such that $s_{\wt\alpha} w$ is $\mu$-permissible, $w < s_{\wt\alpha} w$, and $w$ and $s_{\wt\alpha}w$ have the same translation parts.
\end{prop}

Strictly speaking, the fact that we can choose $\wt\alpha$ such that $s_{\wt\alpha}w$ has the same translation part as $w$ is a stronger conclusion than we need to deduce the Main Theorem in the even case.  On the other hand, this stronger conclusion immediately yields the following corollary, which we will use in the proof of the Main Theorem in the odd case in \s\ref{s:thm_odd}.

\begin{cor}\label{st:=<_trans_part}
If $w \in \wt W_{2n}$ is $\mu$-permissible, then $w$ is less than or equal to its translation part in the Bruhat order.\qed
\end{cor}

\begin{rk}\label{rk:haines}
It should be emphasized here that our notion of translation part is with regard to the decomposition $w = t_{v_0}\sigma$.  But $t_{v_0}\sigma = \sigma t_{\sigma^{-1}v_0}$, so that $t_{\sigma^{-1}v_0}$ is another reasonable definition of the translation part of $w$.  Quite generally, Haines \cite{haines01b}*{proof of 4.6} has shown that for any minuscule cocharacter $\mu$ in any extended affine Weyl group, every $\mu$-admissible element $w = t_{v_0}\sigma$ satisfies $w \leq t_{\sigma^{-1}v_0}$.  The referee points out that this formally implies $w \leq t_{v_0}$ as well:  if $w \in \Adm(\mu)$, then $w^{-1} \in \Adm(-\mu)$, so that $w^{-1} \leq t_{-v_0}$ by Haines's result, so that $w \leq t_{v_0}$.  Thus \eqref{st:=<_trans_part} (or its combination with the Main Theorem) is really a special case of Haines's result.  Haines's proof relies in part on Hecke algebra techniques; it would be interesting to find a more direct proof.
\end{rk}

\begin{proof}[Proof of \eqref{st:strong_thm_even}]
Let $w$ be $\mu$-permissible and not a translation element, let $\sigma$ denote the linear part of $w$, and let $\nu_k = wa_k' - a_k'$ for $0 \leq k \leq n$.  As usual, given an affine root $\wt\alpha$, we denote its linear part by $\alpha$.  Taking $v = \mathbf 0$ in \eqref{disp:wv-v_variants}, we see that $w$ and $s_{\wt\alpha}w$ have the same translation parts $\iff$ $\langle \wt\alpha, \nu_0 \rangle = 0$.

Since $w$ is not a translation element, the various $\nu_k$'s are not all equal, and we let $i$ denote the minimal element in $\{1,\dotsc,n\}$ such that $\nu_0 = \dotsb = \nu_{i-1} \neq \nu_i$.  By the recursion relation \eqref{disp:nu_recursion} we have $\sigma(k) = k$ for all $k <i$, so that by duality $\sigma(k) = k$ for all $k > i^*$ as well.  The common vector $\nu_0 = \dotsb = \nu_{i-1}$ is of the form $e_j - e_{j^*}$ for some $j \in \{1,\dotsc,2n\}$, and, by the recursion relation,
\[
   \nu_i = e_j - e_{j^*} + \frac{e_{\sigma(i)} - e_i + e_{i^*} - e_{\sigma(i)^*}}2.
\]
Since $\sigma(i) \neq i$, the containment $\nu_i \in \Conv(S_{2n}^\circ \mu)$ forces
\begin{equation}\label{disp:j_possibilities}
   j = i \quad \text{or} \quad j = \sigma(i)^*.
\end{equation}
We shall consider these possibilities separately.

First suppose $j = i$, and suppose further that $\nu_r = \mathbf 0$ for some $r$.  By minimality of $i$ we have $i \leq r$; and since $\nu_n \in S_{2n}^\circ \mu$, we have $r < n$.  If $i = r$, then evidently $\sigma(i) = i^*$, and we take $\wt\alpha := \wt\alpha_{i,i^*-1;1}$.  Then $\langle \wt\alpha, \nu_0 \rangle = 0$ and
\[
   \nu_k + \langle \wt\alpha, a_k' \rangle \alpha^\vee =
   \begin{cases}
      e_{i^*-1} - e_{i+1},  &  0 \leq k < i;\\
      -\tfrac 1 2 \alpha^\vee,  &  k = i;\\
      \nu_k,  &  i < k \leq n.
   \end{cases}
\]
Hence $s_{\wt\alpha}w$ is $\mu$-permissible by \eqref{st:perm_cond} and $w < s_{\wt\alpha}w$ by taking $v = a_i'$ in condition \eqref{it:bo1} in \eqref{st:bo}, as desired.

If $i < r$, then we take $\wt\alpha := \wt\alpha_{r,r+1;0}$.  Then $\langle \wt\alpha, \nu_0 \rangle = 0$ and
\[
   \nu_k + \langle \wt\alpha, a_k' \rangle \alpha^\vee =
   \begin{cases}
      \nu_k,  &  k \neq r;\\
      \tfrac 1 2 \alpha^\vee,  &  k = r.
   \end{cases}
\]
Hence $s_{\wt\alpha}w$ is $\mu$-permissible by \eqref{st:perm_cond} and $w < s_{\wt\alpha}w$ by taking $v = a_r'$ in condition \eqref{it:bo1} in \eqref{st:bo}, as desired.

Having dispensed with these cases, we continue to suppose $j = i$ in \eqref{disp:j_possibilities}, but we now suppose that $\nu_k \neq \mathbf 0$ for all $k$.  In particular, $\nu_i \neq \mathbf 0$, so that $\sigma(i) \neq i^*$.  Setting $\wt\imath := \sigma^{-1}(i)$, we obtain $\wt\imath \neq i$, $i^*$.  Hence $i < \wt\imath,\wt\imath^* < i^*$ by minimality of $i$, and we take $\wt\alpha := \wt\alpha_{i,\wt\imath;1}$.  The relation $\langle \wt\alpha, \nu_0 \rangle = 0$ is immediate.  We shall show that $s_{\wt\alpha}w$ is $\mu$-permissible by verifying the condition in \eqref{st:perm_cond} for $k$ separately in the ranges
\[
   0 \leq k < i,\quad i \leq k < m, \quad \text{and}\quad m \leq k \leq n,
\]
where $m := \min\{\wt\imath,\wt\imath^*\}$.

For $0 \leq k < i$, we have $\langle \wt\alpha, a_k'\rangle = -1$ and $\nu_k + \langle \wt\alpha, a_k'\rangle\alpha^\vee = e_{\wt\imath} - e_{\wt\imath^*} \in \Conv(S_{2n}^\circ\mu)$.

In the particular case $k = i$, we have explicitly
\begin{equation}\label{disp:nu_i}
  \nu_i = \frac{e_i - e_{i^*} + e_{\sigma(i)} - e_{\sigma(i)^*}}2,
\end{equation}
and we deduce from the recursion formula \eqref{disp:nu_recursion} and the fact that $\sigma$ is a permutation that $\nu_k(i) = \frac 1 2$ and $\nu_k(i^*) = - \frac 1 2$ for all $i \leq k < m$.  Hence by \eqref{rk:perm_nu} $\nu_k$ is of the form
\begin{equation}\label{disp:special_nu}
   \nu_k = \frac{e_i-e_{i^*} + e_{r_k} - e_{r_k^*}}2
   \quad\text{with}\quad
   r_k \neq i,\ i^*
\end{equation}
for such $k$.  Since $\langle \wt\alpha, a_k'\rangle = - \frac 1 2$ for such $k$, we conclude
\begin{equation}\label{disp:nu_k+<>alpha^vee}
   \nu_k + \langle \wt\alpha, a_k'\rangle\alpha^\vee 
      = \frac{e_{\wt\imath} - e_{\wt\imath^*} + e_{r_k} - e_{r_k^*}}2 
             \in \Conv(S_{2n}^\circ \mu).
\end{equation}

For $m \leq k \leq n$, we have $\langle \wt\alpha, a_k'\rangle = -1$ or $0$ according as $m = \wt\imath$ or $m = \wt\imath^*$.  In the latter case, we obtain immediately $\nu_k + \langle \wt\alpha, a_k'\rangle\alpha^\vee = \nu_k \in \Conv(S_{2n}^\circ \mu)$.  In the former case, we observe
\begin{align}
   \notag
   \nu_{\wt\imath} 
      &= \nu_{\wt\imath - 1} 
         + \frac{e_{i} - e_{\wt\imath} + e_{\wt\imath^*} - e_{i^*}}2
         \quad & &\text{by \eqref{disp:nu_recursion}}\\
      \notag
      &= \frac{e_i-e_{i^*} + e_{r_{\wt\imath - 1}} - e_{r_{\wt\imath - 1}^*}}2
         + \frac{e_{i} - e_{\wt\imath} + e_{\wt\imath^*} - e_{i^*}}2
         \quad & &\text{by \eqref{disp:special_nu}}\\
      \label{disp:nu_wti}
      &= e_{i} - e_{i^*} 
         + \text{(terms not involving $e_i$, $e_{i^*}$)}.
          \quad & &
\end{align}
The condition $\nu_{\wt\imath} \in \Conv(S_{2n}^\circ \mu)$ then forces $\nu_{\wt\imath} = e_i - e_{i^*}$, i.e.\ the terms not involving $e_i$ and $e_{i^*}$ in \eqref{disp:nu_wti} must cancel.  Since $\sigma$ is a permutation, we similarly deduce that $\nu_k = e_i - e_{i^*}$ for all $\wt\imath \leq k \leq n$.  Hence
\[
   \nu_k + \langle \wt\alpha, a_k'\rangle\alpha^\vee 
      = \nu_k - \alpha^\vee
      = e_{\wt\imath} - e_{\wt\imath^*} \in \Conv(S_{2n}^\circ \mu)
\]
for such $k$.  We conclude that $s_{\wt\alpha}w$ is $\mu$-permissible.

To check that $w < s_{\wt\alpha}w$, we consider the cases $m = \wt\imath$ and $m = \wt\imath^*$ separately.  If $m = \wt\imath$, then we take $v = a_{\wt\imath - 1}'$ in condition \eqref{it:bo1} in \eqref{st:bo}.  Indeed, the cancellation of the terms not involving $e_i$ and $e_{i^*}$ in \eqref{disp:nu_wti} implies that
\[
   \nu_{\wt\imath - 1} 
      = \frac{e_i - e_{i^*} + e_{\wt\imath} - e_{\wt\imath^*}}2.
\]
Hence $\langle \alpha, \nu_{\wt\imath - 1} \rangle = 0$ and $\bigl\langle \alpha, \nu_{\wt\imath - 1} + \langle \wt\alpha, a_{\wt\imath - 1}' \rangle \alpha^\vee \bigr\rangle = \langle \alpha, e_{\wt\imath} - e_{\wt\imath^*} \rangle = -1$, and the conclusion follows.  If $m = \wt\imath^*$, then we take $v = a_{\wt\imath^*}'$ in condition \eqref{it:bo2} in \eqref{st:bo}.  Indeed, by the recursion formula and \eqref{disp:special_nu},
\[
   \nu_{\wt\imath^*} = \frac{e_{\wt\imath} - e_{\wt\imath^*} + e_{r_{\wt\imath^* - 1}} - e_{r_{\wt\imath^* - 1}^*}}2
   \quad\text{with}\quad
   r_{\wt\imath^* - 1} \neq i,\ i^*.
\]
Since moreover $\nu_{\wt\imath^*} \neq \mathbf 0$, we have $r_{\wt\imath^* - 1} \neq \wt\imath^*$.  Hence $\langle \wt\alpha, a_{\wt\imath^*}'\rangle = 0$ and $\langle\alpha, \nu_{\wt\imath^*}\rangle < 0$, and the conclusion follows.  This completes the proof in the case $j = i$.

It remains to consider the possibility $j = \sigma(i)^* \neq i$ in \eqref{disp:j_possibilities}.  In this case we have $i < \sigma(i), \sigma(i)^* < i^*$, and we take $\wt\alpha := \wt\alpha_{i,\sigma(i);1}$.  The rest of the proof proceeds very similarly to the case $j = i$ just considered.  The relation $\langle \wt\alpha, \nu_0\rangle = 0$ is immediate. We show that $s_{\wt\alpha}w$ is $\mu$-permissible by again checking the criterion in \eqref{st:perm_cond} for $k$ separately in the ranges $0 \leq k < i$, $i \leq k < m$, and $m \leq k \leq n$, where this time $m := \min\{\sigma(i),\sigma(i)^*\}$.

For $0 \leq k < i$, we have $\langle \wt\alpha, a_k' \rangle = -1$ and $\nu_k + \langle \wt\alpha, a_k' \rangle\alpha^\vee = e_{i^*} - e_i \in \Conv(S_{2n}^\circ \mu)$.

For $k = i$, we have
\[
   \nu_i = \frac{e_{\sigma(i)^*} - e_{\sigma(i)} + e_{i^*} - e_i}2,
\]
and we deduce similarly to the $j = i$ case that $\nu_k$ is of the form
\[
   \nu_k = \frac{e_{\sigma(i)^*} - e_{\sigma(i)} + e_{r_k} - e_{r_k^*}}2
   \quad\text{with}\quad
   r_k \neq \sigma(i),\ \sigma(i)^*
\]
for all $i \leq k < m$.  Since $\langle \wt\alpha, a_k'\rangle = -\frac 1 2$ for such $k$, we conclude
\[
   \nu_k + \langle \wt\alpha, a_k'\rangle\alpha^\vee = \frac{e_{i^*} - e_{i} + e_{r_k} - e_{r_k^*}}2 \in \Conv(S_{2n}^\circ \mu).
\]

For $m \leq k \leq n$, we have $\langle \wt\alpha, a_k'\rangle = -1$ or $0$ according as $m = \sigma(i)$ or $m = \sigma(i)^*$.  In the latter case, we obtain immediately $\nu_k + \langle \wt\alpha, a_k'\rangle\alpha^\vee = \nu_k \in \Conv(S_{2n}^\circ \mu)$.  In the former case, a similar argument to the one given in the $j = i$ case shows that $\nu_k = e_{\sigma(i)^*} - e_{\sigma(i)}$ for all $\sigma(i) \leq k \leq n$, so that $\nu_k + \langle \wt\alpha, a_k'\rangle\alpha^\vee = e_{i^*} - e_i \in \Conv(S_{2n}^\circ \mu)$ for such $k$.  This completes the proof that $s_{\wt\alpha}w$ is $\mu$-permissible.

To see that $w < s_{\wt\alpha}w$, we take $v = a_i'$ in condition \eqref{it:bo1} of \eqref{st:bo}.  Then $\langle \alpha, \nu_i \rangle = 0$ and $\bigl\langle \alpha, \nu_i + \langle \wt\alpha, a_i' \rangle \alpha^\vee \bigr\rangle = -1$, and the conclusion follows.

This completes the proof of \eqref{st:strong_thm_even}, and with it the proof of the Main Theorem for even orthogonal groups.
\end{proof}

\section{Main theorem for odd orthogonal groups}\label{s:thm_odd}

In this section we deduce the Main Theorem for odd orthogonal groups from the Main Theorem for even orthogonal groups.  The key link is the notion of Steinberg fixed-point root datum, which we review in the first subsection.  In contrast with the previous two sections, when considering the groups $\wt W_{2n+1}$ and $\wt W_{2n+2}$ we shall work directly with the vector spaces $\RR^n$ and $\RR^{n+1}$, respectively, instead of embedding these into $\RR^{2n}$ and $\RR^{2n+2}$ as in \eqref{disp:cochar_emb}.

\subsection{Steinberg fixed-point root data}\label{ss:steinberg_fprd}
The main references for this subsection are  the papers of Steinberg \cite{st68}, Kottwitz and Rapoport \cite{kottrap00}, and Haines and Ng\^o \cite{hngo02b}, especially \cite{hngo02b}*{\s9}.  Let $\R = (X^*,X_*,R,R^\vee,\Pi)$ be a reduced and irreducible based root datum.  An \emph{automorphism} $\Theta$ of $\R$ is an automorphism $\Theta$ of the abelian group $X_*$ such that the subsets $R$, $\Pi \subset X^*$ are stable under the dual automorphism $\Theta^*$ of $X^*$ induced by $\Theta$ and the perfect pairing $X^* \times X_* \to \ZZ$.  It follows that any automorphism of $\R$ induces automorphisms of the coroots $R^\vee$, the coroot lattice $Q^\vee$, the Weyl group $W$, the affine Weyl group $W_\aff$, and the extended affine Weyl group $\wt W$.  Associated with $\Theta$ is the \emph{Steinberg fixed-point root datum} $\R^{[\Theta]} = (X^{*[\Theta]},X_*^{[\Theta]},R^{[\Theta]},R^{\vee [\Theta]}, \Pi^{[\Theta]})$; this is a reduced and irreducible based root system described explicitly in \cite{hngo02b}*{\s9}.  We shall just mention here that the cocharacter lattice $X_*^{[\Theta]}$ of $\R^{[\Theta]}$ is the subgroup of $X_*$
\[
   X_*^{[\Theta]} := \{\, x \in X_* \mid \Theta(x) \equiv x \bmod Z \,\},
\]
where $Z := \{\, x \in X_* \mid \langle \alpha, x \rangle = 0\ \text{for all}\ \alpha \in R\,\}$, and that the  Weyl group $W^{[\Theta]}$ of $\R^{[\Theta]}$ is the fixed-point subgroup $W^{\Theta} \subset W$.
 
As described in \s\ref{ss:defs}, the respective extended affine Weyl groups $\wt W$ and $\wt W^{[\Theta]}$ of $\R$ and $\R^{[\Theta]}$ carry canonical Bruhat orders, which we denote $\leq$ and $\leq^{[\Theta]}$.  The key fact for us is that $\leq^{[\Theta]}$ is inherited from $\leq$ under the inclusion $\wt W^{[\Theta]} \subset \wt W$.

\begin{prop}[Kottwitz--Rapoport \cite{kottrap00}*{2.3}, Haines--Ng\^o \cite{hngo02b}*{9.6}]\label{st:bo_inheritance}
Let $x$, $y \in \wt W^{[\Theta]}$.  Then $x \leq^{[\Theta]} y$ in $\wt W^{[\Theta]}$ $\iff$ $x \leq y$ in $\wt W$.\qed
\end{prop}

\subsection{Application to orthogonal groups}
Consider the automorphism of $\RR^{n+1}$
\[
   \Theta \colon (x_1,\dotsc,x_{n+1}) \mapsto (x_1,\dotsc,x_n,-x_{n+1}).
\]
Then $\Theta$ gives an automorphism of the based root datum of $SO_{2n+2}$ in the sense of \s\ref{ss:steinberg_fprd}, and the map
\begin{equation}\label{disp:odd_SO->even_SO}
   (x_1,\dotsc,x_n) \mapsto (x_1,\dotsc,x_n,0)
\end{equation}
realizes the based root datum of $SO_{2n+1}$ as the Steinberg fixed-point root datum.  The induced embedding $\wt W_{2n+1} \inj \wt W_{2n+2}$ sends translation elements to translation elements according to the rule \eqref{disp:odd_SO->even_SO}, and Weyl group elements to Weyl group elements as follows:  if we represent a given $\sigma \in S_{2n}^*$ as  $(\varepsilon_1,\dotsc,\varepsilon_n) \cdot \tau \in \{\pm 1\}^n \rtimes S_n$, then
\[
   (\varepsilon_1,\dotsc,\varepsilon_n) \cdot \tau \mapsto
   (\varepsilon_1,\dotsc,\varepsilon_n,\eta) \cdot \tau' \in \{\pm 1\}^{n+1} \rtimes S_{n+1},
\]
where $\tau'$ is the permutation $\tau'(i) = \tau(i)$ for $1 \leq i \leq n$ and $\tau'(n+1) = n + 1$, and $\eta = 1$ or $-1$ according as $\sigma$ is even or odd.

\subsection{Proof of the theorem}
We now prove the Main Theorem for odd orthogonal groups.  We must show that $\mu$-permissibility implies $\mu$-ad\-miss\-ibility in $\wt W_{2n+1}$; we shall do so by exploiting the validity of this implication in $\wt W_{2n+2}$.  Let $w \in \wt W_{2n+1}$, say with translation part $t_{\nu_0}$ and linear part $\sigma$, and suppose that $w$ is $\mu$-permissible.  Using \eqref{disp:odd_SO->even_SO}, we identify points in $\RR^n$ with their images in $\RR^{n+1}$.  Our permissibility assumption implies that
\[
   \nu_0 \in S_{2n}^*\mu \subset S_{2n+2}^\circ \mu
\]
and that
\[
   wv - v \in \Conv(S_{2n}^*\mu) \subset \Conv(S_{2n+2}^\circ \mu)
\]
for $v$ any of the points
\[
   (0,\dotsc,0), \quad
   (1,0^{(n)}), \quad
   \text{and}\quad 
   \bigl((\tfrac 1 2)^{(i)},0^{(n+1-i)}\bigr)
   \quad \text{for}\quad  2 \leq i \leq n.
\]

We claim that that $w \cdot\mathbf{\frac 1 2} - \mathbf{\frac 1 2} \in \Conv(S_{2n+2}^\circ \mu)$ as well.  Indeed, let $a = \bigl((\tfrac 1 2)^{(n)},0\bigr)$.  Our description of the embedding $S_{2n}^* \inj S_{2n+2}^\circ$ makes plain that
\[
   w \cdot\mathbf{\tfrac 1 2} - \mathbf{\tfrac 1 2} = wa - a
   \quad \text{or} \quad
   wa -a + (0^{(n)},-1)
\]
according as $\sigma$ is even or odd.  If $\sigma$ is even, then our claim follows immediately.  If $\sigma$ is odd, then it follows from \eqref{st:sigma_even} and the containment $wa - a \in \Conv(S_{2n}^*\mu)$ that $wa - a = \mathbf 0$, and our claim again follows.

It is now easy to see that condition \eqref{it:nu_k_perm_cond} in \eqref{st:comb_perm} is satisfied (with $n+1$ in place of $n$), so that $w$ is $\mu$-permissible in $\wt W_{2n+2}$.  The theorem now follows from \eqref{st:=<_trans_part} and \eqref{st:bo_inheritance}.\qed

\begin{rk}
Of course, \eqref{st:=<_trans_part} implies its own analogue here in the odd case, namely that if $w \in \wt W_{2n+1}$ is $\mu$-permissible, then $w$ is less than or equal to its translation part.
\end{rk}

\end{document}